\documentclass[12pt,reqno]{amsart}
\usepackage[utf8]{inputenc}
\usepackage{amsmath}
\usepackage{amssymb,enumitem}
\usepackage{mathtools}
\usepackage{hyperref}
\usepackage{needspace}
\hypersetup{hidelinks}
\usepackage[
  backend=biber,
  style=alphabetic,
  url=false,
  isbn=false
]{biblatex}

\usepackage{ifthen}

\newtheorem {Theorem}    {Theorem}[section]
\newenvironment{Theorem*}
  {\Theorem}
  {\endTheorem}
\newtheorem {Lemma}      [Theorem]    {Lemma}

\newtheorem {Corollary}  [Theorem]    {Corollary}
\newtheorem {Proposition}[Theorem]    {Proposition}

\theoremstyle{definition}
\newtheorem {Definition} [Theorem]    {Definition}

\newtheorem {Remark}    [Theorem]    {Remark}
\newtheorem {Exercise}   {Exercise}
\newenvironment{Exercise*}
  {\Exercise}
  {\endExercise}

\newcounter{AbcT}

\numberwithin{equation}{section}

\newcommand {\R} {{\mathbb R}}

\newcommand {\cU} {{\mathcal U}}

\renewcommand{\limsup}{\varlimsup}
\DeclareMathOperator{\vol}{vol}

\DeclareMathOperator{\mdim}{mdim}

\DeclareMathOperator{\co}{co}

\newcommand {\IGNORE}[1]  {}

\newcommand {\htop}       {{h_{\textup{top}}}}
\renewcommand {\setminus}       {\smallsetminus}

\newcommand{\myentry}[1]{\ifthenelse{\equal{#1}{.}}{}{#1}}

\DeclareMathOperator{\Spoke}{Spoke}

\newcommand{\cW}{{\mathcal W}}

\title[An Entropy Bound for Mean Dimension of Open Covers]
{An Entropy Bound for Mean Dimension of Open Covers}\author{Tal Barak}
\date{\today}
\begin{document}

\begin{abstract}
We establish a counterpart, for individual finite open covers, of the
global entropy-mean-dimension comparison of Lindenstrauss and Weiss.
More precisely, the mean dimension of a finite open cover is
quantitatively bounded in terms of its topological entropy and
cardinality. In particular, a finite open cover with zero entropy has
zero mean dimension. The same comparison holds for actions of countable
amenable groups and in the sofic setting. It follows that every mean
dimension pair is an entropy pair and, in particular, that uniformly
positive mean dimension (UPMD) implies uniformly positive entropy (UPE).
The cover and pair implications answer Questions~2.27 and~4.9 of
Garc\'ia-Ramos and Gutman, respectively. The proof associates a polytope
to a minimal subcover and combines the Lindenstrauss-Weiss compression
argument with Maurey's empirical method to estimate the volumes of its
coordinate projections.
\end{abstract}
\maketitle
\tableofcontents

\section{Introduction}

Mean topological dimension was introduced by Gromov~\cite{Gromov-mean-dimension} as an
invariant for dynamical systems on infinite-dimensional spaces. Lindenstrauss and
Weiss~\cite{Lindenstrauss-Weiss} developed its theory and related it to
topological entropy.
In this paper we compare the two invariants for fixed open covers.

Let $(X,T)$ be a topological dynamical system, i.e.\ $X$ is a nonempty compact
metrizable space and $T\colon X\to X$ is a homeomorphism. Throughout this
paper, all covers are finite. For open covers $\cU$ and $\mathcal V$ of $X$
we write
$\mathcal V\succ\cU$ if every element of $\mathcal V$ is contained in some
element of $\cU$, and we set
$\cU\vee\mathcal V=\{U\cap V:U\in\cU,\ V\in\mathcal V\}$. The order of a cover $\gamma$ is defined by
\[
\operatorname{ord}(\gamma)
=
\max_{x\in X}\#\{U\in\gamma:x\in U\}-1,
\]
and the dimension captured by $\cU$ is defined to be
\[
\mathcal D(\cU)=\min_{\gamma\succ\cU}\operatorname{ord}(\gamma),
\]
where the minimum is over finite open covers $\gamma$ of $X$ refining $\cU$.
Equivalently, $\mathcal D(\cU)$ is the smallest dimension of a polyhedron
$P$ such that there exists a $\cU$-compatible map from $X$ to $P$; for the
definition see Section~\ref{sec:minimal-subcover}.
The dimension $\mathcal D$ is subadditive under joins,
\[ \mathcal D(\cU\vee\mathcal V)\leq\mathcal D(\cU)+\mathcal D(\mathcal V);\]
see~\cite{Lindenstrauss-Weiss}. Hence, writing
$\cU_a^b=\bigvee_{i=a}^bT^{-i}\cU$, the limit
\[
\mdim(X,T,\cU)=\lim_{n\to\infty}\frac{\mathcal D(\cU_0^{n-1})}{n}
\]
exists, and we define the mean dimension of $(X,T)$ by
\[ \mdim(X,T)=\sup_{\cU}\mdim(X,T,\cU),\] where the supremum is taken over the finite open covers of $X$.
Similarly, if $\mathcal N(\cU)$ denotes the minimal cardinality of a
subcover of $\cU$, we define the topological entropy of the cover to be
\[
\htop(X,T,\cU)
=
\lim_{n\to\infty}
\frac{\log\mathcal N(\cU_0^{n-1})}{n},
\]
and \[ \htop(X,T)=\sup_{\cU}\htop(X,T,\cU) \] is the topological entropy of
$(X,T)$, the supremum again being taken over the finite open covers of $X$.
All logarithms in this paper are natural.

Lindenstrauss and Weiss~\cite{Lindenstrauss-Weiss} introduced metric mean
dimension, a metric invariant related to mean dimension, and showed that
$\mdim(X,T)$ is bounded above by the lower metric mean dimension with respect
to every compatible metric. Since metric mean dimension vanishes when the
topological entropy is finite, finite topological entropy implies zero mean
dimension. This comparison was later sharpened: under the marker property
there is a compatible metric for which mean dimension coincides with upper
metric mean dimension~\cite{Lindenstrauss-Tsukamoto-GAFA}.
For the definition and further discussion of metric mean dimension we refer
to~\cite{Lindenstrauss-Weiss}; see also
\cite{Lindenstrauss-Tsukamoto-IEEE} and the introduction of
\cite{Barak-Lindenstrauss-UPMD}.
These results are global and do not give a direct comparison between
$\mdim(X,T,\cU)$ and $\htop(X,T,\cU)$ for a fixed cover $\cU$. At the
level of finite covers one always has the elementary inequality
\[
\mathcal D(\cW)\leq\mathcal N(\cW)-1.
\]
This is too weak for dynamical applications: linear growth of
$\mathcal D(\cU_0^{n-1})$ gives only a linear lower bound on
$\mathcal N(\cU_0^{n-1})$, whereas positive entropy requires exponential
growth. Thus a cover-level comparison strong enough to pass to dynamical
growth rates must involve $\log\mathcal N$ rather than $\mathcal N$ itself.

Throughout this article, a \emph{standard} open cover is a two-element
open cover $\cU=\{U,V\}$ such that neither $U$ nor $V$ is dense in $X$. We say
that $\cU$ \emph{distinguishes} a pair
$(x,y)\in X^2\setminus\Delta_X$ if $x\notin\overline V$ and
$y\notin\overline U$, where $\Delta_X=\{(x,x):x\in X\}$.
\begin{Definition}[Blanchard~\cite{blanchard-upe-cpe,Blanchard-entropy-pairs}]
\leavevmode
\begin{enumerate}[label=\textup{(\arabic*)}]
\item A topological dynamical system $(X,T)$ has \emph{uniformly positive
entropy} (UPE) if $\htop(X,T,\cU)>0$ for every standard open cover $\cU$.
\item A pair $(x,y)\in X^2\setminus\Delta_X$ is called an \emph{entropy pair}
if $\htop(X,T,\cU)>0$ for every standard open cover $\cU$ distinguishing
$(x,y)$.
\end{enumerate}
\end{Definition}

As Blanchard shows in~\cite{Blanchard-entropy-pairs}, $(X,T)$ has UPE if
and only if every pair of distinct points is an entropy pair, and a system has
positive entropy if and only if it has an entropy pair. Later, Kerr and Li
developed the corresponding local entropy theory for actions of sofic groups,
using combinatorial independence~\cite{Kerr-Li-sofic-independence}.

Garc\'ia-Ramos and Gutman introduced local
mean-dimension notions in the generality of sofic group
actions, including sofic UPMD and sofic mean dimension pairs
\cite[Definitions~3.3 and~4.1]{Garcia-Gutman-mdim}. For amenable actions they
also use the classical notion of UPMD
\cite[Definition~3.4]{Garcia-Gutman-mdim}. For $\mathbb Z$-actions we use the
following classical terminology.
\begin{Definition}[Garc\'ia-Ramos-Gutman]
\leavevmode
\begin{enumerate}[label=\textup{(\arabic*)}]
\item We say that a topological dynamical system $(X,T)$ has
\emph{uniformly positive mean dimension} (UPMD) if
$\mdim(X,T,\cU)>0$ for every standard open cover $\cU$.
\item A pair $(x,y)\in X^2\setminus\Delta_X$ is called a
\emph{mean dimension pair} if
$\mdim(X,T,\cU)>0$ for every standard open cover $\cU$ distinguishing
$(x,y)$.
\end{enumerate}
\end{Definition}

Garc\'ia-Ramos and Gutman asked whether zero sofic entropy of a finite open
cover implies zero sofic mean dimension, noting that the question was open
also for $\mathbb Z$-actions, and whether every sofic mean dimension pair is a
sofic entropy pair~\cite[Questions~2.27 and~4.9]{Garcia-Gutman-mdim}.

Independently and around the same time, Elon Lindenstrauss and the author were
led to the same notions and to the two-element case of this problem, in work that later appeared as~\cite{Barak-Lindenstrauss-UPMD}. They also considered the
system-level question~\cite[Question~1.3]{Barak-Lindenstrauss-UPMD}:
does every system with UPMD have UPE? That paper studies hub-and-spoke
systems and shows, among other things, that if a symbolic system $(X,\sigma)$
has UPE, then the associated hub-and-spoke system
$(\Spoke(X),\sigma)$ has UPMD~\cite{Barak-Lindenstrauss-UPMD}.

Our main result is the following quantitative finite-cover comparison.

\begin{Theorem}\label{thm:finite-cover-inequality}
Let $X$ be a nonempty compact metrizable space, let $n\geq1$, and let
$\{ \cU_i \}_{i=0}^{n-1}$ be a family of finite open covers of $X$. Set
$m=\max_i|\cU_i|$, $N=\sum_i|\cU_i|$ and $\cW=\bigvee_{i=0}^{n-1}\cU_i$.
Then
\[
\mathcal D(\cW)^3
\leq
\frac{2\pi e^3+1}{\log 2}\,
m^2N^2\log\mathcal N(\cW)
\leq
\frac{2\pi e^3+1}{\log 2}\,
m^4n^2\log\mathcal N(\cW).
\]
\end{Theorem}

We are not aware of a previous general estimate controlling
$\mathcal D(\cW)$ in terms of $\log\mathcal N(\cW)$ in this way.

Taking $\cU_i=T^{-i}\cU$, we get:

\begin{Corollary}\label{cor:fixed-cover}
Let $(X,T)$ be a topological dynamical system and let $\cU$ be a finite open
cover of $X$. Then
\[
\mdim(X,T,\cU)^3
\leq
\frac{2\pi e^3+1}{\log2}\,
|\cU|^4\,\htop(X,T,\cU).
\]
\end{Corollary}
Thus,
\[
\htop(X,T,\cU)=0
\Longrightarrow
\mdim(X,T,\cU)=0.
\]
\begin{Corollary}\label{cor:pairs}
Every mean dimension pair is an entropy pair. In particular, every system
with UPMD has UPE.
\end{Corollary}

The second statement of Corollary~\ref{cor:pairs} follows from the first,
since a system has UPMD exactly when every pair of distinct points is a mean
dimension pair, and, by Blanchard, has UPE exactly when every such pair is an
entropy pair.

The same conclusions hold for actions of countable amenable groups.
They also hold in the sofic setting. More precisely, if a countable
sofic group $G$ acts continuously on $X$ and $\Sigma$ is a fixed sofic
approximation, then for every finite open cover $\cU$,
\[
\mdim_\Sigma(X,G,\cU)^3
\leq
\frac{2\pi e^3+1}{\log2}\,
|\cU|^4\,h_\Sigma(X,G,\cU).
\]
In particular,
zero sofic entropy of a finite cover implies zero sofic mean dimension, and
every sofic mean dimension pair is a sofic entropy pair; see
Section~\ref{sec:sofic-cover-extension}.

This answers \cite[Questions~2.27 and~4.9]{Garcia-Gutman-mdim}: for
$\mathbb Z$-actions and countable amenable groups by
Corollaries~\ref{cor:fixed-cover} and~\ref{cor:pairs}, and in the original
sofic setting, for any fixed sofic approximation, by
Theorem~\ref{thm:sofic-fixed-cover-comparison} and
Corollary~\ref{cor:sofic-mean-dimension-pairs}.
Combined with~\cite{blanchard-upe-cpe} and
\cite{Barak-Lindenstrauss-UPMD}, Corollary~\ref{cor:pairs} also shows that
UPMD systems are weakly mixing and characterizes UPMD among hub-and-spoke
systems (Section~\ref{sec:consequences}).

The proof of Theorem~\ref{thm:finite-cover-inequality} uses the geometry of a
convex polytope. We start with a minimal subcover of $\cW$, encode its members
by vertices of a cube, and use a partition of unity to map $X$ into their
convex hull $K$. The compression argument of Lindenstrauss and
Weiss~\cite[Theorem~4.2]{Lindenstrauss-Weiss}, applied to this encoding, gives
a lower bound on the total volume of the $d$-dimensional coordinate
projections of $K$, where $d=\mathcal D(\cW)$. Maurey's empirical
method (see Pisier~\cite{Pisier-Maurey} and
\cite[Appetizer]{Vershynin-HDP-2}) bounds the volume of each such projection
in terms of the number of vertices, namely $\mathcal N(\cW)$.
Comparing the two gives the theorem.

Intuitively, the compression argument forces the $d$-dimensional coordinate
projections of $K$ to have total volume at least $m^{-d}$, while the convex hull of $M$ points in a $d$-dimensional cube is
small, since by Maurey's method it is covered by $M^t$ balls of radius about
$\sqrt{d/t}$. To beat the $\binom Nd$ projections we need
$t\approx(mN/d)^2$, and then $\log M\gtrsim d/t\approx d^3/(m^2N^2)$. 
We do not know whether this bound is optimal. 

\medskip

\noindent\textbf{Acknowledgments.}
The author is deeply grateful to Elon Lindenstrauss for introducing him to mean dimension and dynamics, for the many years of teaching, guidance, mathematical conversations, and professional advice that followed, and for helpful discussions about the present work. He also thanks his current advisor, Dorit Aharonov, for her support, encouragement, and generosity in allowing him the freedom to devote some time to questions outside the main direction of his current research.

\medskip

\section{A volume bound for convex hulls}
\label{sec:convex-volume}

Here and below, $\vol_k$ denotes $k$-dimensional Lebesgue measure, and for a
subset $A$ of a real vector space, we write $\co(A)$ for its convex hull.

Lemma~\ref{lem:convex-volume} is a form of Maurey's empirical
method~\cite{Pisier-Maurey},
\cite[Appetizer]{Vershynin-HDP-2}; we include the proof to keep the constant
explicit.

\begin{Lemma}[see~\cite{Matousek-Discrete-Geometry}, Chapter~13]
\label{lem:ball-volume}
Let $k\geq1$ be an integer, and let $B\subseteq\R^k$ be a Euclidean ball
of radius $R>0$. Then
\[
\vol_k(B)
\leq
\left(\frac{2\pi e}{k}\right)^{k/2}R^k.
\]
\end{Lemma}
\begin{Lemma}\label{lem:convex-volume}
Let $k,M\geq1$ be integers, and let $A\subseteq[0,1]^k$ be a nonempty
finite set with $|A|\leq M$. For every integer $t\geq1$,
\begin{equation}\label{eq:convex-volume}
\vol_k(\co(A))
\leq
M^t\left(\frac{C_0}{\sqrt t}\right)^k,
\qquad
C_0=\sqrt{\frac{\pi e}{2}}.
\end{equation}
\end{Lemma}

\begin{proof}
Fix $z\in\co(A)$ and write
\[
z=\sum_{a\in A}\lambda_a a,
\qquad
\lambda_a\geq0,
\qquad
\sum_{a\in A}\lambda_a=1.
\]
Let $V^{(1)},\ldots,V^{(t)}$ be independent random vectors taking values in
$A$, each with distribution $\mathbb P(V^{(\ell)}=a)=\lambda_a$ for $a\in A$,
and set \[Y=\frac1t\sum_{\ell=1}^tV^{(\ell)}.\] Then $\mathbb E [Y]=z$. Since
each coordinate $V_j^{(\ell)}$ takes values in $[0,1]$, we have
$\operatorname{Var}(V_j^{(\ell)})\leq 1/4$. Using independence,
\[
\mathbb E\|Y-z\|_2^2
=
\sum_{j=1}^k\operatorname{Var}(Y_j)
=
\frac1{t^2}\sum_{j=1}^k\sum_{\ell=1}^t
\operatorname{Var}(V_j^{(\ell)})
\leq
\frac{k}{4t}.
\]
Hence for some $v_1,\ldots,v_t\in A$,
\[
\left\|
z-\frac{v_1+\cdots+v_t}{t}
\right\|_2
\leq
\sqrt{\frac{k}{4t}}.
\]
There are at most $|A|^t\leq M^t$ such averages. So $\co(A)$ is covered
by at most $M^t$ Euclidean balls of radius $R=\sqrt{\frac{k}{4t}}$.
By Lemma~\ref{lem:ball-volume}, each such ball has volume at most
\[
\left(\frac{2\pi e}{k}\right)^{k/2}
\left(\frac{k}{4t}\right)^{k/2}
=
\left(\frac{\pi e}{2t}\right)^{k/2}
=
\left(\frac{C_0}{\sqrt t}\right)^k.
\]
Therefore
\[
\vol_k(\co(A))
\leq
M^t\left(\frac{C_0}{\sqrt t}\right)^k.
\]
\end{proof}

\section{The main estimate}\label{sec:finite-cover-comparison}

\subsection{Maps associated to minimal subcovers}\label{sec:minimal-subcover}

Throughout this section we keep the notation of
Theorem~\ref{thm:finite-cover-inequality}: $n\geq1$,
$X$ is a nonempty compact metrizable space,
$\{\cU_i\}_{i=0}^{n-1}$ is a family of finite open covers of $X$, and
$\cW=\bigvee_{i=0}^{n-1}\cU_i$. We recall the following notion from~\cite{Lindenstrauss-Weiss}: for a metrizable compact space $Y$, a
continuous map $f:X\to Y$ is called $\cU$-compatible if there is a finite
open cover $\widetilde{\cU}$ of $f(X)$ such that $f^{-1}(\widetilde{\cU})\succ\cU.$

\begin{Definition}
Let $I$ be a finite set, $E\subseteq I$, and
$\varepsilon=(\varepsilon_i)_{i\in E}\in\{0,1\}^E$. The corresponding face
of $[0,1]^I$ is
\[
F(E,\varepsilon)
=
\left\{
y\in[0,1]^I:
y_i=\varepsilon_i \text{ for every }i\in E
\right\}.
\]
Its dimension is $\dim F(E,\varepsilon)=|I|-|E|$. For $0\leq\ell\leq |I|$, the $\ell$-skeleton of $[0,1]^I$ is the union of
all faces of dimension at most $\ell$.
\end{Definition}

\begin{Proposition}[Lindenstrauss-Weiss~{\cite[Proposition~2.4]{Lindenstrauss-Weiss}}]
\label{prop:compatible-map}
Let $\cU$ be a finite open cover of $X$ and let $k\geq0$ be an integer. Then
$\mathcal D(\cU)\leq k$ if and only if there exists a polyhedron $P$ with
$\dim P\leq k$ and a continuous $\cU$-compatible map $f:X\to P$.
\end{Proposition}

Let $\cU_i=\{U_i^0,\ldots,U_i^{m_i-1}\}$ for $0\leq i<n$, and let
\[ \Lambda=\{(i,j):0\leq i<n,\ 0\leq j<m_i\} .\] For
$a\in\prod_i\{0,\ldots,m_i-1\}$, write $W_a=\bigcap_iU_i^{a_i}$, and define $v^a\in\{0,1\}^{\Lambda}$ by
\[
v^a_{ij}
=
\begin{cases}
1,& j=a_i,\\
0,& j\neq a_i.
\end{cases}
\]
So $v^a$ is the indicator vector of the choice $a$, with exactly one
coordinate equal to $1$ in each block $\{i\}\times\{0,\ldots,m_i-1\}$.

\begin{Lemma}\label{lem:minimal-subcover-map-general}
In the above notation, let $M=\mathcal N(\cW)$. There exist a set
$\mathcal A\subseteq\prod_i\{0,\ldots,m_i-1\}$ with $|\mathcal A|=M$ such
that $\{ W_a : a\in\mathcal A$ \} is a cover of $X$, and a continuous map
\[
f:X\longrightarrow
K=\co\{v^a:a\in\mathcal A\}\subseteq[0,1]^\Lambda
\]
such that, for every $x\in X$ and every $(i,j)\in\Lambda$,
\[
f_{ij}(x)>0\Longrightarrow x\in U_i^j.
\]
Moreover, $\sum_{j=0}^{m_i-1}y_{ij}=1$ for every $y\in K$ and every
$0\leq i<n$.
\end{Lemma}

\begin{proof}
Since $\cW$ consists of the sets $W_a$ and $M=\mathcal N(\cW)$, we may
choose $\mathcal A\subseteq\prod_i\{0,\ldots,m_i-1\}$ with $|\mathcal A|=M$
such that $\{W_a:a\in\mathcal A\}$ covers $X$. Let
$(\varphi_a)_{a\in\mathcal A}$ be a continuous partition of unity
subordinate to this cover, and define
\[
f(x)=\sum_{a\in\mathcal A}\varphi_a(x)v^a.
\]

If $f_{ij}(x)>0$, then some $a\in\mathcal A$ with
$\varphi_a(x)>0$ satisfies $a_i=j$, and hence $x\in W_a\subseteq U_i^j$. Finally, each $v^a$ has exactly one coordinate equal to $1$ in every block
$i$, so $\sum_{j=0}^{m_i-1}y_{ij}=1$ for every $y\in K$ and every
$0\leq i<n$.
\end{proof}

\subsection{The compression argument}\label{sec:compression}

For a finite coordinate set $I$ and $J\subseteq I$, let
$\pi_J\colon[0,1]^I\to[0,1]^J$ denote the coordinate projection, and for $q\in[0,1]^I$ write
$q_J=\pi_J(q)$.

The following is a form of the compression argument used in the proof of
\cite[Theorem~4.2]{Lindenstrauss-Weiss}.
See also \cite[Lemmas~6.4-6.5]{Li-Sofic-Mdim}.

\begin{Lemma}\label{lem:low-dimensional-faces}
Let $I$ be a finite set, let $K\subseteq[0,1]^I$ be nonempty and compact,
and let $1\leq k\leq|I|$. Suppose that $q\in(0,1)^I$ satisfies
\[
q_J\notin\pi_JK
\qquad
\text{for every }J\subseteq I\text{ with }|J|=k.
\]
Then there is a continuous map $r\colon K\to[0,1]^I$ whose image lies in the $(k-1)$-skeleton of $[0,1]^I$, which fixes every
coordinate already equal to $0$ or $1$, and such that, for every $y\in K$
and $i\in I$,
\[
y_i<q_i\Longrightarrow r_i(y)<q_i,
\quad
y_i=q_i\Longrightarrow r_i(y)=q_i,
\quad
y_i>q_i\Longrightarrow r_i(y)>q_i.
\]
\end{Lemma}

\begin{proof}
Set $\Omega=\{y\in[0,1]^I:\#\{i:y_i=q_i\}<k\}$. The hypothesis implies $K\subseteq\Omega$.
Let $F=F(E,\varepsilon)$ be a face of dimension $s\geq k$, and let
$c_F\in F$ be given by
\[
(c_F)_i=
\begin{cases}
\varepsilon_i,& i\in E,\\
q_i,& i\notin E.
\end{cases}
\]
Then $c_F\notin\Omega$, since it agrees with $q$ in all $s\geq k$
coordinates not in $E$.

For each $y\in F\setminus\{c_F\}$, let $r_F(y)$ be the unique point of
$\partial F$ lying on the ray from $c_F$ through $y$. For $i\notin E$,
\[
(r_F(y))_i
=
q_i+\lambda(y)(y_i-q_i),
\qquad
\lambda(y)\geq1,
\]
while the coordinates in $E$ are unchanged. Hence $r_F$ fixes
$\partial F$ and satisfies
\[
\operatorname{sgn}\bigl((r_F(y))_i-q_i\bigr)
=
\operatorname{sgn}(y_i-q_i)
\qquad
(i\in I).
\]
In particular, $r_F(F\cap\Omega)\subseteq\partial F\cap\Omega$.
Now fix $s\geq k$. If two distinct $s$-dimensional faces $F_1$ and $F_2$
intersect, then every point of $F_1\cap F_2$ lies in both boundaries, so
$r_{F_1}(y)=y=r_{F_2}(y)$ for any $y$ in the mutual intersection $F_1\cap F_2\cap\Omega$. Therefore the maps $r_F$, over the $s$-dimensional faces $F$, paste to a
continuous map from the $s$-skeleton intersected with $\Omega$ into the
$(s-1)$-skeleton intersected with $\Omega$. Composing these maps for
$s=|I|,|I|-1,\ldots,k$ and restricting to $K$ gives the required map. If $y_i\in\{0,1\}$, then at each step either $i$ is
a fixed coordinate of the face containing the current point, or the current
point lies in the boundary of that face and is unchanged. Hence that
coordinate remains equal to $y_i$ throughout.
\end{proof}

We now apply the compression argument to the block encoding of
Lemma~\ref{lem:minimal-subcover-map-general}.

\begin{Lemma}\label{lem:coordinate-projections-general}
Let $K$ be as in Lemma~\ref{lem:minimal-subcover-map-general}, set
$m=\max_i m_i$, and assume that $d=\mathcal D(\cW)\geq1$. Then, for every
$q\in(0,1/m)^\Lambda$, there exists $J\subseteq\Lambda$ with $|J|=d$ such
that $q_J\in\pi_JK$.
\end{Lemma}

\begin{proof}
By subadditivity,
\[
d=\mathcal D(\cW)
\leq\sum_{i=0}^{n-1}\mathcal D(\cU_i)
\leq\sum_{i=0}^{n-1}(m_i-1)
\leq N=|\Lambda|.
\]
Suppose to the contrary that there is some $q\in(0,1/m)^\Lambda$ such that
$q_J\notin\pi_JK$ for every $J\subseteq\Lambda$ with $|J|=d$.
Lemma~\ref{lem:low-dimensional-faces}, applied with $I=\Lambda$ and $k=d$,
gives a continuous map $r\colon K\to[0,1]^\Lambda$ whose image lies in the
$(d-1)$\nobreakdash-skeleton and which fixes every coordinate equal to $0$
or~$1$. Set $g=r\circ f$, where $f$ is the map of Lemma~\ref{lem:minimal-subcover-map-general}.
For every $x\in X$ and every $0\leq i<n$, some $j$ satisfies
\[
f_{ij}(x)\geq\frac1{m_i}\geq\frac1m>q_{ij}.
\]
Lemma~\ref{lem:low-dimensional-faces} therefore gives
$g_{ij}(x)>q_{ij}>0$. Thus for every $x\in X$ and every $0\leq i<n$ there exists
$j\in\{0,\ldots,m_i-1\}$ such that $g_{ij}(x)>0$.
Since $r$ fixes every coordinate equal to $0$,
\[
g_{ij}(x)>0
\Longrightarrow
f_{ij}(x)>0
\Longrightarrow
x\in U_i^j.
\]
For $a\in\prod_i\{0,\ldots,m_i-1\}$, let
\[
V_a
=
\{z\in g(X):z_{i,a_i}>0\text{ for every }0\leq i<n\}.
\]
The sets $V_a$ are open in $g(X)$ and cover $g(X)$, and each preimage
$g^{-1}(V_a)$ is contained in the element $\bigcap_iU_i^{a_i}$ of $\cW$.
Hence $g$ is $\cW$-compatible. Since its image lies in the
$(d-1)$\nobreakdash-skeleton, Proposition~\ref{prop:compatible-map} gives
$\mathcal D(\cW)\leq d-1$, a contradiction.
\end{proof}

\begin{Proposition}\label{prop:projection-volume-general}
Let $K$, $d$ and $m$ be as in
Lemma~\ref{lem:coordinate-projections-general}. Then
\[
1
\leq
m^d
\sum_{\substack{J\subseteq\Lambda\\|J|=d}}
\vol_d(\pi_JK).
\]
\end{Proposition}

\begin{proof}
Since $|\Lambda|=N$, set $Q=(0,1/m)^\Lambda$, and for $|J|=d$ let
\[E_J=\{q\in Q:q_J\in\pi_JK\}.\] Lemma~\ref{lem:coordinate-projections-general}
gives $Q\subseteq\bigcup_{|J|=d}E_J$. Moreover, since the coordinates outside $J$ are unrestricted,
\[
\vol_N(E_J)
=
m^{-(N-d)}
\vol_d\bigl(\pi_JK\cap(0,1/m)^J\bigr)
\leq
m^{-(N-d)}\vol_d(\pi_JK).
\]
Since $\vol_N(Q)=m^{-N}$, subadditivity gives
\[
m^{-N}
\leq
m^{-(N-d)}
\sum_{|J|=d}\vol_d(\pi_JK),
\]
and multiplying by $m^N$ gives the claim.
\end{proof}

\subsection{Proof of the main theorem}

\begin{proof}[Proof of Theorem~\ref{thm:finite-cover-inequality}]
Set
\[
d=\mathcal D(\cW),
\qquad
M=\mathcal N(\cW).
\]
If $d=0$ there is nothing to prove, so assume $d\geq1$.

Let $K$ be given by Lemma~\ref{lem:minimal-subcover-map-general}.
For every $J\subseteq\Lambda$ with $|J|=d$, the set $\pi_JK$ is the
convex hull of at most $M$ points of $[0,1]^J$. Hence
Lemma~\ref{lem:convex-volume} gives, for every integer $t\geq1$,
\[
\vol_d(\pi_JK)
\leq
M^t\left(\frac{C_0}{\sqrt t}\right)^d.
\]
There are exactly $\binom Nd$ subsets $J\subseteq\Lambda$ of cardinality
$d$, so summing this bound over all such $J$ and using
Proposition~\ref{prop:projection-volume-general} gives
\[
1
\leq
m^d\sum_{|J|=d}\vol_d(\pi_JK)
\leq
\binom Nd M^t
\left(\frac{mC_0}{\sqrt t}\right)^d.
\]

Now take
\[
t=
\left\lceil
\left(\frac{2eC_0mN}{d}\right)^2
\right\rceil .
\]
Using $\binom Nd\leq(eN/d)^d$ and
\[
\frac{eN}{d}\frac{mC_0}{\sqrt t}
\leq \frac12,
\]
we obtain
\[
1\leq M^t2^{-d},
\qquad
\log M\geq\frac{d\log2}{t}.
\]

Moreover, since $C_0^2=\pi e/2$,
\[
t
\leq
2\pi e^3\frac{m^2N^2}{d^2}+1.
\]
As $d\leq N$ and $m\geq1$, we have $m^2N^2/d^2\geq1$, and hence
\[
t
\leq
(2\pi e^3+1)\frac{m^2N^2}{d^2}.
\]
Therefore
\[
\log M
\geq
\frac{\log2}{2\pi e^3+1}
\frac{d^3}{m^2N^2},
\]
so
\[
d^3
\leq
\frac{2\pi e^3+1}{\log2}\,
m^2N^2\log M.
\]

Since $N\leq mn$, we conclude
\[
\mathcal D(\cW)^3
\leq
\frac{2\pi e^3+1}{\log2}\,
m^2N^2\log\mathcal N(\cW)
\leq
\frac{2\pi e^3+1}{\log2}\,
m^4n^2\log\mathcal N(\cW).
\]
\end{proof}
\section{Dynamical consequences}\label{sec:dynamical-consequences}

\subsection{Classical and amenable actions}

\begin{proof}[Proof of Corollary~\ref{cor:fixed-cover}]
Apply Theorem~\ref{thm:finite-cover-inequality} to
$\cU_i=T^{-i}\cU$, $0\leq i<n$, so that $m\leq|\cU|$; divide by $n^3$ and let
$n\to\infty$.
For amenable groups, argue as in Remark~\ref{rem:amenable}.
\end{proof}

\begin{proof}[Proof of Corollary~\ref{cor:pairs}]
Let $(x,y)$ be a mean dimension pair and let $\cU$ be a standard open cover
distinguishing it. Then $\mdim(X,T,\cU)>0$, so $\htop(X,T,\cU)>0$ by
Corollary~\ref{cor:fixed-cover}. The second statement follows by the same
argument applied to all standard open covers.
\end{proof}

\begin{Remark}\label{rem:amenable}
Let a countable amenable group $G$ act continuously on $X$, and let
$(F_n)$ be a left F\o lner sequence in $G$. For every nonempty finite
$F\subseteq G$, write
\[
\cU^F=\bigvee_{g\in F}g^{-1}\cU.
\]
Applying Theorem~\ref{thm:finite-cover-inequality} to the family
$\{g^{-1}\cU:g\in F_n\}$ gives
\[
\mathcal D(\cU^{F_n})^3
\leq
\frac{2\pi e^3+1}{\log2}\,
|\cU|^4 |F_n|^2\log\mathcal N(\cU^{F_n}).
\]
Hence
\[
\left(\frac{\mathcal D(\cU^{F_n})}{|F_n|}\right)^3
\leq
\frac{2\pi e^3+1}{\log2}\,
|\cU|^4
\frac{\log\mathcal N(\cU^{F_n})}{|F_n|}.
\]
For every finite $F\subseteq G$ and every $h\in G$,
\[
\cU^{Fh}=h^{-1}\cU^F.
\]
Hence the functions
$F\mapsto\mathcal D(\cU^F)$ and
$F\mapsto\log\mathcal N(\cU^F)$
are right-translation invariant. They are also nonnegative, monotone, and
subadditive. By~\cite[Theorem~6.1]{Lindenstrauss-Weiss}, the limits
\begin{align*}
\mdim(X,G,\cU)
&:=\lim_{n\to\infty}\frac{\mathcal D(\cU^{F_n})}{|F_n|},\\
\htop(X,G,\cU)
&:=\lim_{n\to\infty}\frac{\log\mathcal N(\cU^{F_n})}{|F_n|}
\end{align*}
exist and are independent of the chosen left F\o lner sequence.
Letting $n\to\infty$ therefore yields
\[
\mdim(X,G,\cU)^3
\leq
\frac{2\pi e^3+1}{\log2}\,
|\cU|^4\htop(X,G,\cU).
\]
The proof of Corollary~\ref{cor:pairs} applies to these invariants as well.
\end{Remark}

\subsection{The sofic case}\label{sec:sofic-cover-extension}

Let a countable sofic group $G$ act continuously on a nonempty compact
metrizable space $X$, and fix a sofic approximation sequence
$\Sigma=(\sigma_i)_{i\geq1}$ for $G$. Thus each $\sigma_i$ is a map
$G\to\operatorname{Sym}(n_i)$, not necessarily a homomorphism,
$n_i\to\infty$, and for all $g,h\in G$,
\begin{align*}
\frac1{n_i}\,\#\{a:\sigma_i(gh)a=\sigma_i(g)\sigma_i(h)a\}&\to1,\\
\frac1{n_i}\,\#\{a:\sigma_i(g)a\neq\sigma_i(h)a\}&\to1
\qquad\text{if }g\neq h.
\end{align*}
Fix also a compatible metric $\rho$ on $X$.
For a nonempty finite set $F\subseteq G$ and $\delta>0$, the
\emph{model space} $\operatorname{Map}(\rho,F,\delta,\sigma_i)$ is the set of
maps $\varphi\colon\{1,\ldots,n_i\}\to X$ such that
\[
\frac1{n_i}\sum_{a=1}^{n_i}
\rho\bigl(g\varphi(a),\varphi(\sigma_i(g)a)\bigr)^2
\leq\delta^2
\qquad\text{for every }g\in F;
\]
see~\cite[Definition~2.3]{Li-Sofic-Mdim}. Identifying such maps with points
of $X^{n_i}$, the model space is a closed subset of $X^{n_i}$, hence compact.
For a finite open cover $\cU$ of $X$, let $\cU^{n_i}$ be the cover of
$X^{n_i}$ by the products $U_1\times\cdots\times U_{n_i}$ with $U_a\in\cU$,
and let $\cW_i(F,\delta)$ be its restriction to the model space.
Since the model space is closed, computing $\mathcal D(\cW_i(F,\delta))$
using refinements on the model space agrees with the ambient-refinement
convention of~\cite[Remark~2.8]{Garcia-Gutman-mdim}.
Following~\cite[Definition~2.4]{Li-Sofic-Mdim}
and~\cite[Section~2.4]{Garcia-Gutman-mdim}, the fixed-cover sofic mean dimension
and sofic entropy are
\begin{align*}
\mdim_\Sigma(X,G,\cU)
&=\inf_F\inf_{\delta>0}\limsup_{i\to\infty}
\frac{\mathcal D(\cW_i(F,\delta))}{n_i},\\
h_\Sigma(X,G,\cU)
&=\inf_F\inf_{\delta>0}\limsup_{i\to\infty}
\frac{\log\mathcal N(\cW_i(F,\delta))}{n_i},
\end{align*}
where both quotients are taken to be $-\infty$ when the model space is empty.
We use the conventions $(-\infty)^3=-\infty$ and $c(-\infty)=-\infty$
for $c>0$.
By~\cite[Lemma~2.9]{Li-Sofic-Mdim}, $\mdim_\Sigma(X,G,\cU)$ does not depend on
$\rho$. For compatible metrics $\rho,\rho'$, a nonempty finite $F\subseteq G$,
and $\delta>0$, the proof of that lemma gives $\delta'>0$ such that
\[
\operatorname{Map}(\rho,F,\delta',\sigma_i)
\subseteq
\operatorname{Map}(\rho',F,\delta,\sigma_i)
\quad\text{for every }i.
\]
Covering numbers are monotone under restriction, so taking limsups and
infima, and then interchanging the metrics, proves that
$h_\Sigma(X,G,\cU)$ is also independent of $\rho$. Sofic mean
dimension pairs, sofic entropy pairs, sofic UPMD and sofic UPE are defined as
in the introduction, with $\mdim_\Sigma$ and $h_\Sigma$ in place of $\mdim$
and $\htop$; see~\cite[Definitions~2.33-2.34, 3.3 and~4.1]{Garcia-Gutman-mdim}.

\begin{Theorem}\label{thm:sofic-fixed-cover-comparison}
Let $\cU$ be a finite open cover of $X$. Then
\[
\mdim_\Sigma(X,G,\cU)^3
\leq
\frac{2\pi e^3+1}{\log2}\,|\cU|^4\,
h_\Sigma(X,G,\cU).
\]
In particular, $h_\Sigma(X,G,\cU)=0$ implies $\mdim_\Sigma(X,G,\cU)=0$.
\end{Theorem}

\begin{proof}
Write $C=(2\pi e^3+1)/\log2$. Fix $F$ and $\delta$, and let
$Z_i=\operatorname{Map}(\rho,F,\delta,\sigma_i)$. If $Z_i$ is nonempty, then
for each $1\leq a\leq n_i$ the sets $\{\varphi\in Z_i:\varphi(a)\in U\}$,
$U\in\cU$, form an open cover of $Z_i$ with at most $|\cU|$ elements, and the
join of these $n_i$ covers is $\cW_i(F,\delta)$. Applying
Theorem~\ref{thm:finite-cover-inequality} to the compact space $Z_i$ and
dividing by $n_i^3$ gives
\[
\left(\frac{\mathcal D(\cW_i(F,\delta))}{n_i}\right)^3
\leq
C|\cU|^4\,\frac{\log\mathcal N(\cW_i(F,\delta))}{n_i},
\]
and this holds trivially when $Z_i$ is empty. Since $x\mapsto x^3$ is
continuous and increasing on $[-\infty,\infty)$, taking $\limsup_{i\to\infty}$ gives
\begin{align*}
\mdim_\Sigma(X,G,\cU)^3
&\leq
\left(\limsup_{i\to\infty}\frac{\mathcal D(\cW_i(F,\delta))}{n_i}\right)^3\\
&\leq
C|\cU|^4\,\limsup_{i\to\infty}\frac{\log\mathcal N(\cW_i(F,\delta))}{n_i}
\end{align*}
for every $F$ and $\delta$. Taking the infimum over $F$ and $\delta$ on the right gives the
inequality. If $h_\Sigma(X,G,\cU)=0$, then for all $F$ and $\delta$ the model
spaces are nonempty for infinitely many $i$, so
$\mdim_\Sigma(X,G,\cU)\geq0$, and the inequality gives equality.
\end{proof}

\begin{Corollary}\label{cor:sofic-mean-dimension-pairs}
Every sofic mean dimension pair with respect to $\Sigma$ is a sofic entropy
pair with respect to $\Sigma$. In particular, every action with sofic UPMD
has sofic UPE, with respect to the same approximation.
\end{Corollary}

\begin{proof}
As for Corollary~\ref{cor:pairs}, using
Theorem~\ref{thm:sofic-fixed-cover-comparison} in place of
Corollary~\ref{cor:fixed-cover}.
\end{proof}

\section{Further consequences}\label{sec:consequences}

\begin{Remark}
The converse to the first assertion of Corollary~\ref{cor:pairs} does not hold. If
$0<\htop(X,T)<\infty$, then $(X,T)$ has an entropy
pair~\cite{Blanchard-entropy-pairs}, but $\mdim(X,T)=0$, so it has no mean
dimension pair.
\end{Remark}

\begin{Corollary}\label{cor:weak-mixing}
Every topological dynamical system with UPMD is weakly mixing.
\end{Corollary}

\begin{proof}
By Corollary~\ref{cor:pairs}, UPMD implies UPE. Blanchard proved that
every system with UPE is weakly mixing~\cite[Proposition~2]{blanchard-upe-cpe}.
\end{proof}

Recall the hub-and-spoke construction
from~\cite{Barak-Lindenstrauss-UPMD}. Let
$X\subset\{1,\ldots,a\}^{\mathbb Z}$ be a symbolic system, choose distinct
points $p_1,\ldots,p_a$ on the boundary of the unit disk, and define
\[
\Phi:X\times[0,1]^{\mathbb Z}\longrightarrow\overline{\mathbb D}^{\mathbb Z},
\qquad
\Phi(x,t)_i=t_i p_{x_i}.
\]
We write $\Spoke(X)=\Phi(X\times[0,1]^{\mathbb Z})$ and equip it with the
shift $\sigma$.

\begin{Remark}\label{rem:not-strong-mixing}
The conclusion of Corollary~\ref{cor:weak-mixing} cannot in general be
strengthened to strong mixing. Blanchard constructed symbolic systems with UPE
which are not strongly mixing~\cite[Example~5]{blanchard-upe-cpe}. Let
$(X,\sigma)$ be such a system. Then
$\Spoke(X)$ has UPMD by
\cite[Theorem~1.5]{Barak-Lindenstrauss-UPMD}.

Since $X$ is not strongly mixing, there exist nonempty cylinder sets
\[
A=\{x\in X:x_i=a_i,\ i\in F\},
\qquad
B=\{x\in X:x_i=b_i,\ i\in F'\},
\]
with $F,F'\subset\mathbb Z$ finite, and a sequence $n_j\to\infty$ such that
\[
A\cap\sigma^{-n_j}B=\varnothing
\qquad(j\geq1).
\]
Define
\[
\widehat A
=
\{z\in\Spoke(X):
z_i=t_i p_{a_i}\text{ for some }t_i>1/2,\ i\in F\},
\]
and define $\widehat B$ analogously. These sets are nonempty, and they are
open in $\Spoke(X)$ because distinct spokes meet only at $0$.

For every $n\in\mathbb Z$,
\[
\widehat A\cap\sigma^{-n}\widehat B\neq\varnothing
\quad\Longleftrightarrow\quad
A\cap\sigma^{-n}B\neq\varnothing.
\]
Indeed, if $z=\Phi(x,t)\in\widehat A\cap\sigma^{-n}\widehat B$, the radii of
$z$ at the coordinates in $F$ and in $F'+n$ are positive, so they determine the
labels: $x_i=a_i$ for $i\in F$ and $x_{i+n}=b_i$ for $i\in F'$, i.e.\
$x\in A\cap\sigma^{-n}B$. Conversely, if $x\in A\cap\sigma^{-n}B$, then
$\Phi(x,(1)_{i\in\mathbb Z})\in\widehat A\cap\sigma^{-n}\widehat B$.
Hence
\[
\widehat A\cap\sigma^{-n_j}\widehat B=\varnothing
\qquad(j\geq1),
\]
so $\Spoke(X)$ is not strongly mixing.
\end{Remark}
\begin{Corollary}\label{cor:spoke-characterization}
Let $(X,\sigma)$ be a nonempty symbolic system. The following are equivalent:
\begin{enumerate}[label=\textup{(\arabic*)}]
\item $(X,\sigma)$ has UPE;
\item $(\Spoke(X),\sigma)$ has UPMD;
\item $(\Spoke(X),\sigma)$ has UPE.
\end{enumerate}
\end{Corollary}

\begin{proof}
The implication (1)$\Rightarrow$(2) is
\cite[Theorem~1.5]{Barak-Lindenstrauss-UPMD}. The implication
(2)$\Rightarrow$(3) is Corollary~\ref{cor:pairs}.
The implication (3)$\Rightarrow$(1) was noted after Theorem~1.5 in
\cite{Barak-Lindenstrauss-UPMD}; we include a proof.

Suppose now that $(X,\sigma)$ does not have UPE, and let $\{U,V\}$ be a
standard open cover of $X$ with $\htop(X,\sigma,\{U,V\})=0$. Since cylinder
sets form a base for the topology of $X$, there are nonempty cylinder sets
\[
A=\{x\in X:x_i=a_i,\ i\in E\},
\qquad
B=\{x\in X:x_i=b_i,\ i\in E'\}
\]
with $A\cap U=\varnothing$ and $B\cap V=\varnothing$. Then
$A\cap B=\varnothing$, so $\mathcal V'=\{X\setminus A,X\setminus B\}$ is an
open cover of $X$, and $\{U,V\}\succ\mathcal V'$. Hence
$\htop(X,\sigma,\mathcal V')=0$. Let
\[
A^*=\{z\in\Spoke(X):z_i=t_ip_{a_i}\text{ for some }t_i\geq1/2,\ i\in E\},
\]
and define $B^*$ analogously. These are closed sets, and they contain the
nonempty open sets $\widehat A$ and $\widehat B$ defined as in
Remark~\ref{rem:not-strong-mixing}. As in that remark, if $z=\Phi(x,t)$ and
$\sigma^kz\in A^*$, then $\sigma^kx\in A$, and similarly for $B$. In
particular $A^*\cap B^*=\varnothing$, so
$\mathcal V=\{\Spoke(X)\setminus A^*,\Spoke(X)\setminus B^*\}$ is a standard
open cover of $\Spoke(X)$. If
$x\in\bigcap_{k=0}^{n-1}\sigma^{-k}(X\setminus C_k)$ with
$C_k\in\{A,B\}$, then
$\Phi(x,t)\in\bigcap_{k=0}^{n-1}\sigma^{-k}(\Spoke(X)\setminus C_k^*)$ for
every $t$. Hence
$\mathcal N(\mathcal V_0^{n-1})\leq\mathcal N({\mathcal V'}_0^{n-1})$ for
every $n$, so $\htop(\Spoke(X),\sigma,\mathcal V)=0$ and
$(\Spoke(X),\sigma)$ does not have UPE. This proves (3)$\Rightarrow$(1).
\end{proof}

\medskip

\Needspace{6\baselineskip}
\medskip
\noindent\textbf{AI-assisted tools.}
ChatGPT (OpenAI), using the GPT-5.6 Sol model, was used for LaTeX editing,
phrasing suggestions, and as an auxiliary tool for checking the arguments.
The mathematical results, proof strategy, and arguments were developed by the
author. The author independently verified the mathematical content and takes
full responsibility for the paper.

\printbibliography
\end{document}